\documentclass[11pt]{article}

\usepackage{latexsym} 
\usepackage{theorem}
\usepackage{graphics}
\usepackage{graphicx}
\usepackage{amsmath}
\usepackage{amssymb}

\newtheorem{defeng}{Definition}[section]
\newtheorem{theorem}[defeng]{Theorem}

\newtheorem{lemma}[defeng]{Lemma}

\newtheorem{conjecture}[defeng]{Conjecture}
\newtheorem{corollary}[defeng]{Corollary}
{\theorembodyfont{\rmfamily} }
{\theorembodyfont{\rmfamily} }
{\theorembodyfont{\rmfamily} }
{\theoremstyle{break}\theorembodyfont{\rmfamily} }
{\theoremstyle{break}\theorembodyfont{\rmfamily} }

\newcounter{claim}

\newenvironment{proof}[1][]%
 {\noindent {\setcounter{claim}{0}\sc proof ---
   }{#1}{}}{\hfill$\Box$\vspace{2ex}} 

\newenvironment{claim}[1][]%
{\refstepcounter{claim}\vspace{1ex}\noindent{(\it\arabic{claim}){#1}{}}\it}{\vspace{1ex}}

\newenvironment{proofclaim}[1][]%
	{\noindent {}{#1}{}}{ This proves~(\arabic{claim}).\vspace{1ex}}

\newcommand{\sm}{\setminus} 

\title{On triangle-free graphs that do not contain a subdivision of
  the complete graph on four vertices as an induced subgraph}

\author{Nicolas Trotignon\thanks{CNRS, LIP, ENS
    de Lyon, INRIA, Universit\'e de Lyon.
    Partially supported by ANR project Stint under reference ANR-13-BS02-0007}~ and Kristina Vu\v{s}kovi\'c\thanks{School of
    Computing, University of Leeds, Leeds LS2 9JT, UK; and Faculty of Computer Science
    (RAF), Union University, Knez Mihajlova 6/VI, 11000 Belgrade, Serbia.  Partially supported by
    EPSRC grant EP/K016423/1 and Serbian Ministry of Education and
    Science projects 174033 and III44006.}}
    
\begin{document}

\maketitle

\begin{abstract}
  We prove a decomposition theorem for the class of triangle-free
  graphs that do not contain a subdivision of the complete graph on
  four vertices as an induced subgraph.  We prove that every graph
   of girth at least~5 in this class is 3-colorable.

{\bf\noindent AMS Classification: } 05C75
\end{abstract}

\section{Introduction}

Here graphs are simple and finite.  We say that $G$ \emph{contains}
$H$ when $H$ is isomorphic to an induced subgraph of $G$. We say that
a graph $G$ is {\em $F$-free} if $G$ does not contain $F$. For a
family of graphs ${\cal F}$, we say that $G$ is {\em ${\cal F}$-free}
if for every $F\in {\cal F}$, $G$ does not contain $F$.
\emph{Subdividing} an edge $e=vw$ of a graph $G$ means deleting $e$,
and adding a new vertex $u$ of degree~2 adjacent to $v$ and $w$.  A
\emph{subdivision} of a graph $G$ is any graph $H$ obtained from $G$
by repeatedly subdividing edges.  Note that $G$ is a subdivision of
$G$.  We say that $H$ is an \emph{ISK4 of a graph $G$} when $H$ is an
induced subgraph of $G$ and $H$ is a subdivision of $K_4$ (where $K_4$
denotes the complete graph on four vertices).  ISK4 stands for
``Induced Subdivision of $K_4$''.

In~\cite{nicolas:isk4}, a decomposition theorem for ISK4-free graphs
is given (see Theorem~\ref{th:decISK4}) and it is proved that their
chromatic number is bounded by a constant $c$.  The proof
in~\cite{nicolas:isk4} follows from a theorem by K\"uhn and
Osthus~\cite{kuhnOsthus:04}, from which it follows that $c$ is at
least $2^{512}$.  It is conjectured in~\cite{nicolas:isk4} that every
ISK4-free is 4-colorable.  The goal of this paper is to prove a
stronger decomposition theorem for ISK4-free graphs, under the
additional assumption that they are triangle-free (see
Theorems~\ref{decomp} and~\ref{maindecomp}).  We also propose the
following conjectures, prove that the first one implies the second one
(see~Theorem~\ref{th:imply}), and prove both of them for graphs of
girth at least~5 (see Theorem~\ref{th:color}).  A complete bipartite
graph with partitions of size $|V_1|=m$ and $|V_2|=n$, is denoted
by $K_{m,n}$.

\begin{conjecture}
  \label{conj1}
  Every \{triangle, ISK4, $K_{3, 3}$\}-free graph contains a vertex of
  degree at most~2. 
\end{conjecture}

\begin{conjecture}
  \label{conj2} 
  Every \{triangle, ISK4\}-free graph is 3-colorable. 
\end{conjecture}

In Section~\ref{sec:dec}, we state several known decomposition
theorems, and derive easy consequences of them for our class. In
particular we prove that Conjecture~\ref{conj1} implies
Conjecture~\ref{conj2}.  In Section~\ref{sec:proofDec}, we prove our
main decomposition theorem.  In Section~\ref{sec:chordless}, we give
some properties needed later for the class of chordless graphs (graphs
where all cycles are chordless).  In Section~\ref{sec:deg2}, we prove
Conjecture~\ref{conj1} for graphs of girth at least~5.

\section{Decomposition theorems}
\label{sec:dec}

In this section, we provide the notation needed to state 
decomposition theorems for ISK4-free graphs, and we state them.

A graph $G$ is \emph{series-parallel} if no \emph{subgraph} of $G$ is
a subdivision of $K_4$.  Clearly, every series-parallel graph is
ISK4-free.  When $R$ is a graph, the \emph{line graph} of $R$ is the
graph $G$ whose vertex-set is $E(R)$ and such that two vertices of $G$
are adjacent whenever the corresponding edges are adjacent in~$R$. 

For a graph $G$, when $C$ is a subset of $V(G)$, we write $G\sm C$
instead of $G[V(G) \sm C]$.  A \emph{cutset} in a graph $G$ is a set
$C$ of vertices such that $G \sm C$ is disconnected.  A \emph{star
  cutset} is a cutset $C$ that contains a vertex $c$, called a
\emph{center} of $C$, adjacent to all other vertices of $C$.  Note
that a star cutset may have more than one center, and that a cutset of
size~1 is a star cutset.  A \emph{double star cutset} is a cutset $C$
that contains two adjacent vertices $x$ and $y$, such that every
vertex of $C \sm \{x, y\}$ is adjacent to $x$ or $y$.  Note that a
star cutset of size at least~2 is a double star cutset.

A path from a vertex $a$ to a vertex $b$ is refered to as an
\emph{$ab$-path}.  A \emph {proper 2-cutset} of a connected graph
$G=(V,E)$ is a pair of non-adjacent vertices $a, b$ such that $V$ can
be partitioned into non-empty sets $X$, $Y$ and $\{ a,b \}$ so that:
there is no edge between $X$ and $Y$; and both $G[X \cup \{ a,b \}]$
and $G[Y \cup \{ a,b \}]$ contain an $ab$-path and neither of $G[X
\cup \{a,b \}]$ nor $G[Y \cup \{ a,b \}]$ is a chordless path.  We say
that $(X, Y, a, b)$ is a \emph{split} of this proper 2-cutset.  The
following is the main decomposition theorem for ISK4-free graphs.

\begin{theorem}[see \cite{nicolas:isk4}]
  \label{th:decISK4}
  If $G$ is an ISK4-free graph, then $G$ is series-parallel, or $G$ is
  the line graph of a graph of maximum degree at most 3, or $G$ has a
  proper 2-cutset, a star cutset, or a double star cutset.
\end{theorem}

Our first goal is to improve this theorem for triangle-free graphs.
Some improvements are easy to obtain (they trivially follow from the
absence of triangles).  The non-trivial one is done in the next two
sections: we show that double star cutsets and proper 2-cutsets are in
fact not needed.  We state the result now, but it follows from
Theorem~\ref{maindecomp} that needs more terminology and is slightly
stronger.

\begin{theorem}\label{decomp}
If $G$ is a \{triangle, ISK4\}-free graph, then either $G$ is 
a series-parallel graph or a complete bipartite graph, or $G$ has a
clique cutset of size at most two, or $G$ has a star cutset. 
\end{theorem}

\begin{proof}
  Follows directly from Theorem~\ref{maindecomp}. 
\end{proof}

We state now several lemmas from \cite{nicolas:isk4} that we need.  A
\emph{hole} in a graph is a chordless cycle of length at least~4.  A
\emph{prism} is a graph made of three vertex disjoint paths of length
at least~1, $P_1 = a_1\dots b_1$, $P_2 = a_2\dots b_2$ and $P_3 =
a_3\dots b_3$, with no edges between them except the following:
$a_1a_2$, $a_1a_3$, $a_2a_3$, $b_1b_2$, $b_1b_3$, $b_2b_3$.  Note that
the union of any two of the paths of a prism induces a hole.  A {\em
  wheel $(H,x)$} is a graph that consists of a hole $H$ plus a vertex
$x\not\in V(H)$ that has at least three neighbors on $H$.

\begin{lemma}[see \cite{nicolas:isk4}]
  \label{l:begin}
  If $G$ is an ISK4-free graph, then either $G$ is a series-parallel
  graph,  or $G$ contains a prism, or $G$ contains a wheel or $G$
  contains $K_{3, 3}$. 
\end{lemma}

A \emph{complete tripartite graph} is a graph that can be partitioned
into three stable sets so that every pair of vertices from two different
stable sets is an edge of the graph.

\begin{lemma}[L\'ev\^eque, Maffray and Trotignon \cite{nicolas:isk4}]
  \label{l:decK33}
  If $G$ is an ISK4-free graph that contains $K_{3, 3}$, then either
  $G$ is a complete bipartite graph, or $G$ is a complete tripartite 
  graph, or $G$ has a clique-cutset of size at most~3.
\end{lemma}

We now state  the consequences of the lemmas above for triangle-free
graphs.  

\begin{lemma}
  \label{l:beginT}
  If $G$ is a \{triangle, ISK4\}-free graph, then either $G$ is a
  series-parallel graph, or $G$ contains a wheel or $G$ contains
  $K_{3, 3}$.
\end{lemma} 

\begin{proof}
  Clear from Lemma~\ref{l:begin} and the fact that every prism
  contains a triangle.
\end{proof}

\begin{lemma}
  \label{l:decK33T}
  If $G$ is an \{ISK4, triangle\}-free graph that contains $K_{3,
    3}$, then either $G$ is a complete bipartite graph, or $G$ has a
  clique-cutset of size at most~2.
\end{lemma}

\begin{proof}
  Clear from Lemma~\ref{l:decK33} and the fact that complete
  tripartite graphs and clique-cutsets of size~3 contain triangles.
\end{proof}

It is now easy to prove the next theorem. 

\begin{theorem}
  \label{th:imply}
  If Conjecture~\ref{conj1} is true, then Conjecture~\ref{conj2} is
  true. 
\end{theorem}

\begin{proof}
  Suppose that Conjecture~\ref{conj1} is true.  Let $G$ be a
  \{triangle, ISK4\}-free graph.  We prove Conjecture~\ref{conj2} by
  induction on $|V(G)|$.  If $|V(G)| = 1$, the outcome is clearly
  true.  

  If $G$ contains $K_{3, 3}$, then by Lemma~\ref{l:decK33T},
  either $G$ is bipartite and therefore 3-colorable, or $G$ has a
  clique-cutset $K$.  In this last case, we recover a 3-coloring of
  $G$ from 3-colorings of $G[K\cup C_1]$, \dots, $G[K\cup C_k]$ where
  $C_1, \dots, C_k$ are the connected components of $G\sm K$.  

  If $G$ contains no $K_{3, 3}$, then by Conjecture~\ref{conj1} it has
  a vertex $v$ of degree at most~2.  By the induction hypothesis, $G
  \sm \{v\}$ has a 3-coloring, and we 3-color $G$ by giving to $v$ a
  color not used by its two neighbors.
\end{proof}

\section{Proof of the decomposition theorem}
\label{sec:proofDec}

\subsection*{Appendices to a hole}

When $x$ is a vertex of a graph $G$, $N(x)$ denotes the neighborhood of
$x$, that is the set of all vertices of $G$ adjacent to $x$.  We set
$N[x] = N(x) \cup \{x\}$.  When $C\subseteq V(G)$, we set $N(C) =
(\cup_{x\in C} N(x)) \sm C$.  When $G$ is a graph, $K$ an induced
subgraph of $G$, and $C$ a set of vertices disjoint from $V(K)$, the
\emph{attachment} of $C$ to $K$ is $N(C) \cap V(K)$, that we also
denote by $N_K(C)$.

When $P=p_1 \dots p_k$ is a path and $1\leq i,  j \leq k$, we
denote by  $p_i P p_j$ the $p_ip_j$-subpath of $P$. 
Let $A$ and $B$ be two disjoint vertex sets such that no vertex of $A$ is
adjacent to a vertex of $B$. A path $P=p_1 \ldots p_k$ {\em connects
$A$ and $B$} if either $k=1$ and $p_1$ has a neighbor in $A$ and a 
neighbor in $B$, or $k>1$ and one of the two endvertices of $P$ is adjacent
to at least one vertex in $A$ and the other is adjacent to at least one vertex
in $B$. $P$ is a {\em direct connection between $A$ and $B$} if in
$G[V(P) \cup A \cup B]$ no path connecting $A$ and $B$ is shorter than $P$.
The direct connection $P$ is said to be {\em from $A$ to $B$}
if $p_1$ is adjacent to a vertex of $A$ and $p_k$ is adjacent to a vertex
of $B$.

Let $H$ be a hole.  A chordless path $P=p_1 \ldots p_k$ in $G
\setminus H$ is an {\em appendix} of $H$ if no vertex of $P \setminus \{
p_1,p_k\}$ has a neighbor in $H$, and one of the following holds:
\begin{itemize}
\item[(i)] $k=1$, $N(p_1) \cap H=\{ u_1,u_2 \}$ and $u_1u_2$ is not an
edge, or
\item[(ii)] $k>1$, $N(p_1) \cap H=\{ u_1\}$, 
$N(p_k) \cap H=\{ u_2\}$ 
and $u_1\neq u_2$.
\end{itemize}

So $\{ u_1,u_2 \}$ is an attachment of $P$ to $H$. The two $u_1u_2$-subpaths
of $H$ are called the {\em sectors} of $H$ w.r.t. $P$.

Let $Q$ be another appendix of $H$, with attachment $\{ v_1,v_2 \}$.
Appendices $P$ and $Q$ are {\em crossing} if one sector of $H$ w.r.t. $P$
contains $v_1$, the other contains $v_2$ and 
$\{ u_1,u_2 \} \cap \{ v_1,v_2\} = \emptyset$.

\begin{lemma}\label{ca}
  If $G$ is an \{ISK4, $K_{3,3}$\}-free graph, then no two appendices
  of a hole of $G$ can be crossing.
\end{lemma}

\begin{proof}
Let $P=p_1\ldots p_k$ and $Q=q_1 \ldots q_l$ be appendices of a hole
$H$ of $G$, and suppose that they are crossing. 
Let $\{ u_1,u_2\}$ be the attachment of $P$ to $H$, and let
$\{ v_1,v_2\}$ be the attachment of $Q$ to $H$. 
So $\{ u_1,u_2\} \cap \{ v_1,v_2 \}=\emptyset$ and w.l.o.g. 
$u_1,v_1,u_2,v_2$ appear in this order when traversing $H$.
W.l.o.g. $u_1$ is
adjacent to $p_1$, and $v_1$ to $q_1$. 

A vertex of $P$ must be adjacent to or coincident with a vertex of $Q$,
since otherwise $H \cup P \cup Q$ induces an ISK4.
Note that $\{ p_1,p_k \} \cap \{ q_1,q_l\} =\emptyset$.
Let $p_i$ be the vertex of $P$ with lowest index that has a neighbor in $Q$,
and let $q_j$ (resp. $q_{j'}$) 
be the vertex of $Q$ with lowest (resp. highest) index adjacent to $p_i$.
Note that $p_i$ is not coincident with a vertex of $Q$.

First suppose that $i=k$. If $j \neq l$ then $H \cup P \cup \{ q_1, \ldots
,q_j \}$ induces an ISK4. So $j=l$. In particular, no vertex of $P$ is
coincident with a vertex of $Q$, and $p_kq_l$ is the only edge between $P$
and $Q$. If $k \neq 1$ then $H \cup Q \cup \{ p_k \}$ induces an ISK4.
So $k=1$ and by symmetry $l=1$. Since $H \cup \{ p_1,q_1 \}$ cannot
induce a $K_{3,3}$, w.l.o.g. $u_1v_1$ is not an edge. Let $H'$ be the
$u_1v_1$-subpath of $H$ that does not contain $u_2$ and $v_2$.
Then $(H \setminus H') \cup \{ u_1,v_1,p_1,q_1\}$ induces an ISK4.
Therefore $i<k$.

If $p_i$ has a unique neighbor in $Q$, then $H \cup Q \cup \{ p_1, \ldots
,p_i \}$ induces an ISK4. Let $H_Q'$ be the sector of $H$ w.r.t. $Q$ that
contains $u_1$. If $p_i$ has exactly two neighbors in $Q$, then
$H'_Q \cup Q \cup \{ p_1, \ldots ,p_i \}$ induces an ISK4. So $p_i$
has at least three neighbors in $Q$. In particular,
$j'\not\in \{ j,j+1 \}$. But then 
$H \cup \{ p_1, \ldots ,p_i,q_1, \ldots,q_j,q_{j'}, \ldots ,q_l \}$
induces an ISK4.
\end{proof}

\begin{lemma}\label{ca2}
  Let $G$ be an \{triangle, ISK4, $K_{3,3}$\}-free graph.  Let $H$ be
  a hole and $P=p_1 \ldots p_k$, $k>1$, a chordless path in $G \sm
  H$. Suppose that $|N(p_1) \cap H|=1$ or 2, $|N(p_k) \cap H|=1$ or 2,
  no vertex of $P \sm \{ p_1,p_k \}$ has a neighbor in $H$, $N(p_1)
  \cap H \not \subseteq N(p_k) \cap H$ and $N(p_k) \cap H \not
  \subseteq N(p_1) \cap H$.  Then $P$ is an appendix of $H$.
\end{lemma}

\begin{proof}
  Assume not. If one of $p_1$ or $p_k$ has one neighbor in $H$ and the
  other one has two neighbors in $H$, then $H \cup P$ induces an ISK4.
  So $|N(p_1) \cap H|=|N(p_k) \cap H|=2$, and hence (since $G$ is
  triangle-free) both $p_1$ and $p_k$ are appendices of $H$. By Lemma
  \ref{ca}, $p_1$ and $p_k$ cannot be crossing.  So for a sector $H'$
  of $H$ w.r.t. $p_1$, $N(p_k) \cap H \subseteq H'$.  But then $H'
  \cup P$ induces an ISK4.
\end{proof}

\subsection*{Wheels}

Let $(H,x)$ be a wheel contained in a graph $G$.
A \emph{sector} is a subpath of $H$ whose
endvertices are adjacent to $x$ and interior vertices are not.  
Two sectors are {\em consecutive} or {\em adjacent} 
if they have an endvertex in
common.  

Throughout this section we use the following notation for a wheel
$(H,x)$. We denote by $x_1, \ldots ,x_n$ the neighbors of $x$ in $H$,
appearing in this order when traversing $H$.  In this case, we also
say that $(H, x)$ is an \emph{$n$-wheel}.  For $i=1, \ldots ,n$, $S_i$
denotes the sector of $(H,x)$ whose endvertices are $x_i$ and
$x_{i+1}$ (here and throughout this section we assume that indices are
taken modulo $n$).

A path $P$ is an {\em appendix} of a wheel $(H,x)$ if the following
hold:
\begin{itemize}
\item[(i)] $P$ is an appendix of $H$,
\item[(ii)] each of the sectors of $H$ w.r.t. $P$ properly contains
a sector of $(H,x)$, and
\item[(iii)] $x$ has at most one neighbor in $P$.
\end{itemize}

\begin{lemma}\label{wa1}
Let $G$ be an ISK4-free graph. 
Let $P$ be an appendix of a wheel $(H,x)$ of $G$, and let $H'_P$ be
a sector of $H$ w.r.t. $P$. Then $H'_P$ contains at least three 
neighbors of $x$.
In particular, $H'_P$ contains at least two sectors of $(H,x)$.
\end{lemma}

\begin{proof}
Let $\{ u_1,u_2 \}$ be the attachment of $P$ to $H$. Since $P$ is an
appendix of $(H,x)$, $H'_P$ contains at least two neighbors of $x$.
Suppose $H'_P$ contains exactly two neighbors of $x$. If $x$ has a neighbor 
in $P$, then $H'_P \cup P \cup \{ x \}$ induces an ISK4. So $x$ does not
have a neighbor in $P$. Since $P$ is an appendix of $(H,x)$, $H'_P$
properly contains a sector of $(H,x)$ and so w.l.o.g. $x$ is not adjacent to
$u_2$. Let $H''_P$ be the other sector of $H$ w.r.t. $P$, and let $x'$
be the neighbor of $x$ in $H''_P$ that is closest to $u_2$. Note that
since $H \cup \{ x \}$ cannot induce an ISK4, $n \geq 4$, and hence
$x' \neq u_1$ and $x'u_1$ is not an edge. Let $H'$ be the $x'u_2$-subpath
of $H''_P$. Then $H'_P \cup H' \cup P \cup \{ x \}$ induces an ISK4.
\end{proof}

A wheel  $(H, x)$ of $G$ is {\em proper} if vertices 
$u \in G \sm (H \cup N[x])$ are  one
of the following types:

  \begin{itemize}
  \item
    type~0: $|N(u) \cap H|=0$; 
  \item
    type~1: $|N(u) \cap H|=1$; 
  \item
    type~2: $|N(u) \cap H|=2$ and for some sector
$S_i$ of $(H,x)$,  
$N(u) \cap H \subseteq S_i$. 
  \end{itemize}

\begin{lemma}\label{w1}
Let $G$ be a \{triangle, ISK4\}-free graph.
If $(H,x)$ is a wheel of $G$ with fewest number of vertices, then
$(H,x)$ is a proper wheel.
\end{lemma}

\begin{proof}
Let $u \in G \sm (H \cup N[x])$. It follows from the following two
claims that $u$ is of type 0, 1 or 2 w.r.t. $(H,x)$, and hence that
$(H,x)$ is proper.

  \begin{claim}
    \label{w1-1}
For every sector $S_i$ of $(H,x)$, $|N(u) \cap S_i|\leq 2$.
  \end{claim}
  
  \begin{proofclaim}
Otherwise, $S_i \cup \{ x,u \}$ induces a wheel with fewer vertices than
$(H,x)$, a contradiction.
  \end{proofclaim}

  \begin{claim}
    \label{w1-2}
    For some sector $S_i$ of $(H,x)$, $N(u) \cap H \subseteq N(u) \cap S_i$.
  \end{claim}
  
  \begin{proofclaim}
    Assume otherwise, and choose $i,j \in \{ 1, \ldots ,n \}$ so that
    $u$ has a neighbor in $S_i \sm S_j$ and in $S_j \sm S_i$,
    $N(u)\cap (S_i\cup S_j)$ is not contained in a sector of $(H,x)$
    and $|j-i|$ is minimized.  W.l.o.g. $i=1$ and $1<j<n$ (since the
    case when $j=n$ is symmetric to the case when $j=2$).  Let $u'$
    (resp. $u''$) be the neighbor of $u$ in $S_1$ that is closest to
    $x_1$ (resp. $x_2$).  Let $u_j$ (resp. $u_{j+1}$) be the neighbor
    of $u$ in $S_j$ that is closest to $x_j$ (resp. $x_{j+1}$).  Let
    $P$ be the $u''u_j$-subpath of $H$ that contains $x_2$. Note that
    by the choice of $i,j$, vertex $u$ has no neighbor in the interior
    of $P$ and $u_j \neq x_n$. Since $G$ is triangle-free,
    $x_{j+1}x_1$ is not an edge, and if $j \neq 2$ then $x_2x_j$ is
    not an edge.

First suppose that $u$ has at least two neighbors in $S_1$. 
Then by (\ref{w1-1}), $u$ has exactly two neighbors in $S_1$.
If $u_{j+1}=x_j$ 
then $j \neq 2$ (by the choice of $i,j$), and hence $S_1 \cup \{ x,u,x_j\}$
induces an ISK4. 
So $u_{j+1} \neq x_j$. If $u_{j+1}x_2$ is an edge, then 
$ux_2$ is not an edge (since $G$ is triangle-free) and hence
$S_1 \cup \{ x,u,u_{j+1}\}$ induces an ISK4. 
So $u_{j+1}x_2$ is not an edge.
But then $S_1 \cup \{ x,u \}$ together with $u_{j+1}S_jx_{j+1}$ induces an
ISK4.

Therefore $u$ has exactly one neighbor in $S_1$, and by symmetry it has
exactly one neighbor in $S_j$. If $j=2$ then $S_1 \cup S_2 \cup \{ x,u \}$
induces an ISK4. 
So $j>2$. If $P$ contains at least three neighbors of $x$, then (since 
$j<n$ and $u_j \neq x_n$) $P \cup \{ x,u \}$ induces a wheel with center $x$
that has fewer vertices than $(H,x)$, a contradiction. Therefore $P$
contains exactly two neighbors of $x$. But then $j=3,u_j \neq x_4$ and hence
$S_1 \cup S_2 \cup \{ x,u \}$ together with $x_3S_3u_j$ induces an ISK4.
  \end{proofclaim}

\end{proof}

\begin{lemma}\label{w2}
  Let $G$ be a \{triangle, ISK4, $K_{3, 3}$\}-free graph.  Let $(H,x)$
  be a proper wheel of $G$ with fewest number of spokes. If $(H,x)$
  has an appendix, then $(H,x)$ is a 4-wheel.
\end{lemma}

\begin{proof}
Assume $(H,x)$ has an appendix. Note that by Lemma \ref{wa1}, if $P$ is
an appendix of $(H,x)$ then each of the sectors of $H$ w.r.t. $P$ contains at 
least two sectors of $(H,x)$. If $(H,x)$ has an appendix such that one of
the sectors of $H$ w.r.t. this appendix is $S_i \cup S_{i+1}$ for some
$i \in \{ 1, \ldots ,n \}$, then let $P$ be such an appendix and
$H'_P=S_i \cup S_{i+1}$. Otherwise, let $P$ be an appendix of $(H,x)$
such that for a sector $H'_P$ of $H$ w.r.t. $P$, there is no appendix
$Q$ of $(H,x)$ such that $H'_P$ properly contains a sector of $H$ w.r.t.
$Q$. Assume further that such a $P$ is chosen so that $|V(P)|$ is
minimized. Let $H''_P$ be the other sector of $H$ w.r.t. $P$. W.l.o.g.
we may assume that $H'_P$ contains $x_1, \ldots ,x_l$, $l \geq 3$, and
does not contain $x_{l+1}, \ldots ,x_n$. Let $\{ y_1,y_2 \}$ be the
attachment of $P$ to $H$ such that $y_1 \in S_n \sm x_n$ and $y_2 \in
S_l \sm x_{l+1}$.
Let $P=p_1 \ldots p_k$ and w.l.o.g. assume that $p_1$ is adjacent to $y_1$,
and $p_k$ to $y_2$. Let $S'_n$ be the $x_1y_1$-subpath of $S_n$, and 
$S'_l$ the $x_ly_2$-subpath of $S_l$.
Let $H'$ be the hole induced by $H'_P\cup P$. Note that since $l \geq 3$,
$(H',x)$ is a wheel.

  \begin{claim}
    \label{w2-1}
$(H',x)$ is a proper wheel.
  \end{claim}
  
  \begin{proofclaim}
Let $u \in G \sm (H' \cup N[x])$ and assume that $u$ is not of type 0, 1 or 2
w.r.t. $(H',x)$. 
Note that $u \not \in H$.
Since $(H,x)$ is a proper wheel, $u$ must have a neighbor
in $P$. Let $P'=y_1Py_2$.
Let $u_1$ (resp. $u_2$) be the neighbor of $u$ in $P$ that is closest to
$y_1$ (resp. $y_2$). Note that sectors $S_1, \ldots ,S_{l-1}$ of
$(H,x)$ are also sectors of $(H',x)$.

First suppose that $N(u) \cap H' \subseteq S'_n \cup S'_l \cup P$.
Since $(H,x)$ is proper, $u$ cannot have a neighbor in both $S'_n$ and $S'_l$.
If $u$ has a neighbor in $S'_n \sm y_1$, then 
by Lemma~\ref{ca2} and since $u$ has a neighbor in $P$, $P \cup u$
contains an appendix $Q$ of $(H,x)$ such that a sector of $H$ w.r.t. $Q$
is properly contained in $H'_P$, contradicting our choice of $P$. So
$u$ does not have a neighbor in
$S'_n \sm y_1$, and by symmetry it does not have a neighbor in
$S'_l \sm y_2$.
Since $u$ is not of type 0, 1 or 2 w.r.t. $(H',x)$, $u_1 \neq u_2$, and 
since $G$ is triangle-free, $u_1u_2$ is not an edge. Let $P''$ be the
chordless path from $y_1$ to $y_2$ in $P \cup u$ that contains $u$.
Since the length of $P''$ cannot be less than the length of $P$, by the choice
of $P$, there is  a vertex $p \in P$ such that $u_1p$ and $u_2p$ are edges.
Since $G$ is triangle-free, $u$ is not adjacent to $p$, and hence $u$ has
exactly two neighbors in $P'$. Since $u$ is not of type 2 w.r.t. $(H',x)$,
$x$ must be adjacent to $p$. But then $S'_n \cup S'_l \cup P \cup \{ x,u \}$
induces an ISK4.

Therefore $u$ must have a neighbor in 
$(S_1 \cup \ldots \cup S_{l-1}) \sm \{ x_1,x_l \}$. Since $(H,x)$ is
proper, for some $i \in \{ 1, \ldots ,l-1 \}$, $N(u) \cap H \subseteq S_i$.
Suppose that $u$ is of type 2 w.r.t. $(H,x)$. W.l.o.g. $u$ is not adjacent
to $x_1$. Let $p_j$ be the vertex of $P$ with lowest index adjacent to $u$.
If $j=k$ then, since $G$ is triangle-free, not both $p_k$ and $u$ can be 
adjacent to $x_l$, and hence $H \cup \{ u,p_k \}$ induces an ISK4.
So $j < k$, and hence $H \cup \{ u,p_1, \ldots ,p_j \}$ induces an ISK4.
Therefore $u$ is of type 1 w.r.t. $(H,x)$.
Recall that by our assumption that $u$ has a neighbor in
$(S_1 \cup \ldots \cup S_{l-1}) \sm \{ x_1,x_l \}$, $u$ is not adjacent to 
$x_1$ nor $x_l$. But then $H \cup P \cup \{ u \}$ contains an ISK4.
  \end{proofclaim}

So $(H',x)$ is a proper wheel. Since it cannot have fewer sectors than $(H,x)$
and by Lemma \ref{wa1},
it follows that $y_1=x_1$, $y_2=x_l$ and $l=n-1$. But then by the choice of
$P$, $l=3$, and hence $(H,x)$ is a 4-wheel.
\end{proof}

A {\em short connection} between sectors $S_i$ and $S_{i+1}$ of $(H,x)$
is a chordless path $P=p_1 \ldots p_k$, $k>1$, in 
$G \setminus (H \cup N[x])$ such that
the following hold:
\begin{itemize}
\item[(i)] 
$N(p_1) \cap (H \setminus \{ x_{i+1}\})=\{ u_1\}$, $u_1 \in S_i
\setminus \{ x_{i+1} \}$,
\item[(ii)] $N(p_k) \cap (H \setminus \{ x_{i+1}\})=\{ u_2\}$, $u_2 \in S_{i+1}
\setminus \{ x_{i+1} \}$, and
\item[(iii)] the only vertex of $H$ that may have a neighbor in
$P \setminus \{ p_1,p_k \}$ is $x_{i+1}$.
\end{itemize}

\begin{lemma}\label{w3}
Let $G$ be a \{triangle, ISK4, $K_{3,3}$\}-free graph.
Let $(H,x)$ be a proper wheel of $G$ with fewest number of spokes.
 Then $(H,x)$ has no short connection.
\end{lemma}

\begin{proof}
Suppose $(H,x)$ has a short connection $P=p_1 \ldots p_k$. 
Assume that $(H,x)$ and $P$ are chosen so that $|V(P)|$ is minimized.
W.l.o.g. $p_1$ is adjacent to $u_1 \in S_1 \sm x_2$ and $p_k$ to
$u_2 \in S_2 \sm x_2$. Let $S_1'$ be the $u_1x_1$-subpath of $S_1$, and let
$S_2'$ be the $u_2x_3$-subpath of $S_2$. Let $S_P$ be the $u_1u_2$-subpath
of $H$ that contains $x_2$. Let $H'$ be the hole induced by
$(H \sm S_P) \cup P \cup \{ u_1,u_2 \}$.

  \begin{claim}
    \label{w3-1}
$n \geq 5$
  \end{claim}
  
  \begin{proofclaim}
If $n=3$ then $H \cup \{ x \}$ induces an ISK4.
If $n=4$ then $H' \cup \{ x \}$ induces an ISK4.
Therefore, $n \geq 5$.
  \end{proofclaim}

  \begin{claim}
    \label{w3-2}
$(H,x)$ has no appendix.
  \end{claim}
  
  \begin{proofclaim}
Follows from (\ref{w3-1}) and Lemma \ref{w2}.
  \end{proofclaim}

  \begin{claim}
    \label{w3-3}
Vertex $x_2$ has a neighbor in $P \sm \{ p_1,p_k \}$.
  \end{claim}
  
  \begin{proofclaim}
    Assume not. If $x_2$ has no neighbor in $P$, then $S_1 \cup S_2
    \cup P \cup \{x\}$ induces an ISK4. So w.l.o.g. $x_2$ is adjacent
    to $p_1$. Then $u_1x_2$ is not an edge, since $G$ is
    triangle-free. If $x_2$ is not adjacent to $p_k$, then $S'_1 \cup
    S_2 \cup P \cup \{x\}$ induces an ISK4. So $x_2$ is adjacent to $p_k$, and
    hence $u_2x_2$ is not an edge. But then $S'_1 \cup S'_2 \cup P
    \cup \{ x,x_2\}$ induces an ISK4.
  \end{proofclaim}

  \begin{claim}
    \label{w3-4}
$(H',x)$ is a proper wheel.
  \end{claim}
  
  \begin{proofclaim}
By (\ref{w3-1}) $(H',x)$ is a wheel. Assume it is not proper 
and let $u \in G \sm (H' \cup N[x])$ be such that it is not of
type 0, 1 or 2 w.r.t. $(H',x)$. Note that $u \not \in H$, and hence
$u$ is of type 0, 1 or 2 w.r.t. $(H,x)$. It follows that $u$ must have
a neighbor in $P$. Let $p_i$ (resp. $p_j$) be the vertex of $P$ with
lowest (resp. highest) index adjacent to $u$.

First suppose that $N(u) \cap H' \subseteq S'_1 \cup S'_2 \cup P$.
Then $|N(u) \cap H'| \geq 3$. If $u$ does not have a neighbor in
$(S_1 \cup S_2 ) \sm \{ x_2 \}$, then $u$ has at least three neighbors in $P$
and hence $p_1Pp_iup_jPp_k$ is a short connection of $(H,x)$ that
contradicts our choice of $P$. So $u$ has a neighbor in 
$(S_1 \cup S_2) \sm \{ x_2 \}$. W.l.o.g. 
$N(u) \cap (S_1 \sm \{ x_2 \}) \neq \emptyset$.
Since $(H,x)$ is proper, $N(u) \cap (S_2 \sm \{ x_2\} ) = \emptyset$.
If $u$ has a unique neighbor in $S_1$, then $j>2$ 
(since $G$ is triangle-free and $u$ has at least two neighbors in $P$)
and hence $up_jPp_k$
is a short connection of $(H,x)$ that contradicts our choice of $P$.
So $u$ is of type 2 w.r.t. $(H,x)$.
If $j=1$ then $|N(u) \cap H'| = 3$ and hence $H' \cup \{ u \}$ induces an ISK4. So $j >1$.
If $ux_2$ is an edge then 
$j>2$ (since $G$ is triangle-free and $u$ has at least three neighbors in
$H'$) and hence
$up_jPp_k$ is a short connection of $(H,x)$
that contradicts our choice of $P$. So $ux_2$ is not an edge.
If $x_2$ has no neighbor in $p_jPp_k$ and $u_2 =x_3$, then let 
$Q=p_jPp_kx_3$. Otherwise let $Q$ be the chordless path from
$p_j$ to $x_2$ in $(S_2 \sm S'_2) \cup \{ p_j, \ldots ,p_k,u_2 \}$.
Then $S_1 \cup Q \cup \{ x,u \}$ induces an ISK4.

Therefore, for some $l \in \{ 3, \ldots ,n \}$, $u$ has a neighbor in
$S_l \sm \{ x_1,x_3 \}$. Since $(H,x)$ is proper, $N(u) \cap
H\subseteq S_l$. If $j=1$ let $Q=up_1$, if $i=k$ let $Q=up_k$, and
otherwise let $Q$ be a chordless path in $(P \sm \{ p_1,p_k \}) \cup
\{ u \}$ from $u$ to a vertex of $P \sm \{ p_1,p_k \}$ that is
adjacent to $x_2$ (note that such a vertex exists by (\ref{w3-3})).
By Lemma \ref{ca2} $Q$ is an appendix of $H$. In particular, $u$ is of
type 1 w.r.t. $(H,x)$. Let $u'$ be the neighbor of $u$ in $H$.  Since
by (\ref{w3-2}) $Q$ cannot be an appendix of $(H,x)$, w.l.o.g.  $j=1$
and $l=n$. Note that $u' \neq x_1$ and $u_1\neq x_2$. Let $p_t$ be the
vertex of $P$ with lowest index adjacent to $x_2$ (such a vertex
exists by (\ref{w3-2})). If $u'=x_n$ then $S_1\cup \{
x,u,u',p_1,\ldots ,p_t\}$ induces an ISK4.  So $u' \neq x_n$. If If
$u_1\neq x_1$ then $S_1\cup \{ x,u,p_1,\ldots ,p_t\}$ together with
the $x_1u'$-subpath of $S_1$ induces an ISK4. So $u_1 =x_1$.  But then
$S_n \cup \{ x,u,p_1, \ldots ,p_t, x_2\}$ induces an ISK4.
  \end{proofclaim}

By (\ref{w3-4}) $(H',x)$ is a proper wheel that has fewer sectors than $(H,x)$,
a contradiction.
\end{proof}

\begin{lemma}\label{w4}
Let $G$ be a \{triangle, ISK4, $K_{3,3}$\}-free graph.
Let $(H,x)$ be a proper wheel of $G$ with fewest number of spokes.
 Then for every $i \in \{ 1, \ldots ,n \}$, $(N[x] \sm H) \cup 
\{ x_i,x_{i+1}\}$ is a star cutset separating $S_i$ from $H \setminus S_i$. 
\end{lemma}

\begin{proof}
Assume not. Then w.l.o.g. there is a direct connection $P=p_1 \ldots p_k$
from $S_1 \sm \{ x_1,x_2 \}$ to $H \sm S_1$ in 
$G \sm ((N[x] \sm H) \cup \{ x_1,x_2 \})$. Note that the only vertices
of $H$ that may have a neighbor in the interior of $P$ are $x_1$ and $x_2$.
Since $(H,x)$ is proper, 
$k>1$ and
$p_1$ and $p_k$ are of type 1 or 2 w.r.t. $(H,x)$.
Let $i \in \{ 2, \ldots ,n \}$ be such that $N(p_k) \cap H \subseteq S_i$.

First suppose that no vertex of $\{ x_1,x_2 \}$ has a neighbor in
$P \sm \{ p_1,p_k \}$. By Lemma \ref{ca2}, $P$ is an appendix of $H$.
In particular, $p_1$ and $p_k$ are both of type 1 w.r.t. $(H,x)$.
If $i \not \in \{ 2,n \}$ then $P$ is an appendix of $(H,x)$.
It follows from Lemma \ref{w2} that $n=4$ and $i=3$. But then,
since $p_1$ has a neighbor in $S_1 \sm \{ x_1,x_2 \}$, $(H,x)$ and $P$
contradict Lemma \ref{wa1}.
So $i \in \{ 2,n \}$, and hence $P$ is a short connection, contradicting
Lemma \ref{w3}. Therefore, a vertex of $\{ x_1,x_2 \}$ has a neighbor
in $P \sm \{ p_1,p_k \}$.

Let $p_j$ be the vertex of $P \sm p_1$ with highest index adjacent to
a vertex of $\{ x_1,x_2 \}$. W.l.o.g. $p_jx_2$ is an edge. 
We now show that if $p_1$ has two neighbors in $S_1 \sm x_2$, then
$x_1$ has a neighbor in $P \sm p_1$. Assume not. Then $p_1$ has two
neighbors in $S_1 \sm x_2$ and $x_1$ does not have a neighbor in
$P \sm p_1$. Since $p_1$ is of type 2 w.r.t. $(H,x)$, $p_1x_2$
is not an edge. But then $H$ and $p_1 \ldots p_{j'}$ (where
$p_{j'}$ is the vertex of $P$ with lowest index adjacent to $x_2$)
contradict Lemma \ref{ca2}.

Suppose $i=2$. If $p_k$ has two neighbors in $S_2 \sm x_2$, then $S_2
\cup \{x, p_j, \dots, p_k\}$ induces an ISK4. So, $p_k$ has a unique
neighbor in $S_2 \sm x_2$.  If $x_1$ does not have a neighbor in $P \sm p_1$,
then $p_1$ has a unique neighbor in $S_1 \sm x_2$ and hence 
$P$ is a short connection of $(H,x)$, contradicting Lemma \ref{w3}.
So $x_1$ has a neighbor in $P \sm p_1$. Let $p_t$ be such a neighbor 
with highest index. Then $p_tPp_k$ is a short connection of $(H,x)$,
contradicting Lemma \ref{w3}. So $i \neq 2$. 
If $i=n$ then either $p_jPp_k$ is a short connection of $(H,x)$
contradicting Lemma \ref{w3}, or $S_n \cup \{x, x_2, p_j, \dots,
p_k\}$ contains an ISK4.
 
Therefore, $i \in \{ 3,\ldots ,n-1 \}$ and $p_k$ has a neighbor in
$H \sm (S_1 \cup S_2 \cup S_n)$. By Lemma \ref{ca2} applied to $H$ and 
$p_jPp_k$, vertex $p_k$ is of type 1 w.r.t. $(H,x)$ and $x_1p_j$ is not an
edge. But then $p_jPp_k$ is an appendix of $(H,x)$. By Lemma \ref{w2}
it follows that $n=4$ and $i=3$. But then, since $p_k$ has a neighbor in
$S_3 \sm \{ x_3,x_4 \}$, $(H,x)$ and $p_jPp_k$ contradict Lemma \ref{wa1}.
\end{proof}

We say that a graph $G$ has a {\em wheel decomposition} if for
some wheel $(H,x)$, for every $i  \in \{ 1, \ldots ,n \}$,
$(N[x] \sm H) \cup \{ x_i,x_{i+1} \}$ is a cutset separating
$S_i$ from $H \sm S_i$. We say that such a wheel decomposition is
{\em w.r.t. wheel $(H,x)$}.  Note that if a graph has a wheel
decomposition, then it has a star cutset.

\begin{theorem}\label{maindecomp}
If $G$ is a \{triangle, ISK4\}-free graph, then either $G$ is 
a series-parallel graph or a complete bipartite graph, or $G$ has a
clique cutset of size at most two, or $G$ has a wheel decomposition.
\end{theorem}

\begin{proof}
Assume $G$ is not series-parallel nor a complete bipartite graph.
By Lemma \ref{l:beginT} $G$ contains a wheel or $K_{3,3}$.
By Lemma \ref{l:decK33T} if $G$ contains a $K_{3,3}$ then it has
a clique cutset of size at most two. So we may assume that $G$
does not contain a $K_{3,3}$. So $G$ contains a wheel. By Lemma~\ref{w1}
$G$ contains a proper wheel, and hence by Lemma~\ref{w4}
$G$ has a wheel decomposition.
\end{proof}

\begin{theorem}\label{maindecomp2}
  If $G$ is a \{triangle, ISK4, $K_{3,3}$\}-free graph, then either
  $G$ is series-parallel or $G$ has a wheel decomposition.
\end{theorem}

\begin{proof}
Assume $G$ is not series-parallel.
By Lemma \ref{l:beginT} $G$ contains a wheel.
By Lemma \ref{w1}
$G$ contains a proper wheel, and hence by Lemma \ref{w4}
$G$ has a wheel decomposition.
\end{proof}

The following corollary is needed in the next section. 

\begin{corollary}\label{cmain}
If $G$ is an ISK4-free graph of girth at least 5, then either $G$ is series-parallel
or $G$ has a star cutset.
\end{corollary}

\begin{proof}
Follows directly from Theorem \ref{maindecomp2} because $K_{3, 3}$
contains a cycle of length~4.
\end{proof}

\section{Chordless graphs}
\label{sec:chordless}

A graph $G$ is \emph{chordless} if no cycle in $G$ has a chord.
Chordless graphs were introduced in~\cite{nicolas:isk4} as roots of
wheel-free line graphs, and it is a surprise to us that we need them
here for a completely different reason in a very similar class.  A
graph is \emph{sparse} if every edge is incident to at least one
vertex of degree at most~2.  A sparse graph is chordless because any
chord of a cycle is an edge between two vertices of degree at least
three.  Recall that proper 2-cutsets are defined in
Section~\ref{sec:dec}.

\begin{theorem}[see~\cite{aboulkerRTV:propeller}]
  \label{t:chordless}
  If $G$ is a 2-connected chordless graph, then either $G$ is sparse
  or $G$ admits a proper 2-cutset.
\end{theorem}

The following theorem is mentioned in \cite{aboulkerRTV:propeller}
without a proof, and we need it in the next section.  So, we prove it
for the sake of completeness.

\begin{theorem}
  \label{th:Cchordless}
  In every cycle of a 2-connected chordless graph that is not a cycle,
  there exist four vertices $a,b,c,d$ that appear in this order and
  such that $a, c$ have degree~2 and b, d have degree at least~3.
\end{theorem}

\begin{proof}
  We prove the result by induction on $|V(G)|$.  If $G$ is sparse
  (in particular, if $|V(G)| = 3$), then it is enough to check that
  every cycle of $G$ contains at least two vertices of degree at least~3,
  because these vertices cannot be adjacent in a sparse graph.  But
  this true, because a cycle with all vertices of degree 2 must be the
  whole graph (since $G$ is connected), and a cycle with a unique
  vertex of degree~3 cannot exists in a 2-connected graph (the vertex
  of degree at least~3 would be a cut-vertex). 

  So, by Theorem~\ref{t:chordless} we may assume that $G$ has a proper
  2-cutset with split $(X, Y, u, v)$.  We now build two blocks of
  decompositions of $G$ as follows.  The block $G_X$ is obtained from
  $G[X \cup \{u, v\}]$ by adding a marker vertex $m_Y$ adjacent to $u$
  and $v$, and the block $G_X$ is obtained from $G[Y \cup \{u, v\}]$
  by adding a marker vertex $m_X$ adjacent to $u$ and $v$.  By the
  definition of proper 2-cutsets, $|V(G_X)|, |V(G_Y)| \leq |V(G)|$.
  Also, $G_X$ and $G_Y$ are chordless and 2-connected.  So we may
  apply the induction hypothesis to the blocks of decomposition.

  Let $C$ be a cycle of $G$.  If $V(C) \subseteq X\cup \{u, v\}$, then
  $C$ is a cycle of $G_X$, so by the induction hypothesis we get four
  vertices $a, b, c, d$ in $C$.  We now check that a vertex $w \in V(C)$
  has degree 2 in $G$ if and only if it has degree 2 in $G_X$.  This is
  obvious, except if $w \in \{u, v\}$.  But in this case, because of
  $m_Y$ and because $w$ lies in a cycle of $G_X$ that does not contain
  $m_Y$, $w$ has degree at least~3 in both $G$ and $G_X$.  This proves
  our claim.  It follows that we obtain by the induction hypothesis
  the condition that we need for the degrees of $a$, $b$, $c$ and $d$.
  The proof is similar when $V(C) \subseteq Y\cup \{u, v\}$. 

  We may now assume that $C$ has vertices in $X$ and $Y$.  It follows
  that $C$ edge wise partitions into a path $P = u\dots v$ whose
  interior is in $X$ and a path $Q = u \dots v$ whose interior is in
  $Y$.  We apply the induction hypothesis to $G_X$ and $C_X = u P v
  m_Y u$. So, we get four vertices $a$, $b$, $c$ and $d$ in $C_X$ and
  they have degree 2, $\geq 3$, 2, $\geq 3$ respectively (in $G_X$).
  These four vertices are in $C$ and have the degrees we need (in
  $G$), except possibly when $|\{a, c\} \cap \{u, m_Y, v\}| = 1$.  In
  this case, we may assume w.l.o.g. that $a=u$ or $a=m_Y$, and we find
  in place of $a$ a vertex of degree~2 in $Y \cap V(C_Y)$, where $C_Y
  = u Q v m_X u$.  This vertex exists by the induction hypothesis
  applied to $G_Y$ and $C_Y$.
\end{proof}

\section{Degree 2 vertices}
\label{sec:deg2}

We need the following application of Menger's theorem. 

\begin{lemma}
  \label{l:tree}
  Let $T$ be a tree, and suppose that the vertices of $T$ are labelled
  with labels $x$ and $y$ (each vertex may receive one label, both
  labels, or no label).   One and only one of the following situations
  occures.
  \begin{itemize}
  \item In $T$, there exist two vertex-disjoint paths $P$ and $Q$, and
    each of them is from a vertex with label $x$ to a vertex with
    label $y$ (possibly, $P$ and/or $Q$ have length~0).
  \item There exists $v \in V(T)$ and two subtrees of $T$, $T_x$ and
    $T_y$ such that:
    \begin{enumerate}
    \item $V(T_x) \cup V(T_y) = V(T)$;
    \item $V(T_x) \cap V(T_y) = \{v\}$;
    \item $T_x$ contains all vertices of $T$ with label $x$ and $T_y$
      contains all vertices of $T$ with label $y$.
    \end{enumerate}
  \end{itemize}
\end{lemma}

\begin{proof}
  It is clear that not both outcomes hold, because if the second
  holds, then $v$ must be on every path from a vertex with label $x$
  to a vertex with label $y$, so no two such paths can be
  vertex-disjoint.  

  If at most one vertex of $T$ has label $x$, then we nominate this
  vertex as $v$ and we set $T_x=T[\{v\}]$, $T_y = T$.  It is easy
  to see that $v,$ $T_x$ and $T_y$ satisfy all the requirements of the
  second outcome.  Hence, we may assume that at least two vertices of
  $T$ have label $x$, and similarly, at least two vertices of $T$ have
  label $y$.  By the classical Menger's theorem, if the first outcome does
  not hold, then there exists a vertex $v$ in $T$ such that every path
  from a vertex with label $x$ to a vertex with label $y$ contains
  $v$.  In particular, every component of $T\sm v$ contains at most
  one label.  So, the second outcome holds: we define $T_x$ as the
  subtree of $T$ formed by $v$ and all components containing vertices
  with label $x$, and we define $T_y$ as the subtree of $T$ formed by
  $v$ and all the other components.
\end{proof}

A graph $G$ together with two of its vertices $x$ and $y$ such that
$xy\in E(G)$ or $x=y$, have the \emph{$(x, y)$-property} if $V(G)\sm
(N[x] \cup N[y])$ contains a vertex of degree 2 in $G$.  Instead of
$(x,x)$-property, we simply write \emph{$x$-property}.  The $(x,
y)$-property is very convenient for us, because it ensures the
existence of vertices of degree~2, and also because it is well
preserved in proofs by induction.  Unfortunately, not all graphs in
our class have the $(x, y)$-property, for intance the graphs
represented in Fig.~\ref{fig} do not have the $x$-property when $x$ is
a vertex with maximum degree in the graph.

\begin{figure}
  \begin{center}
  \includegraphics{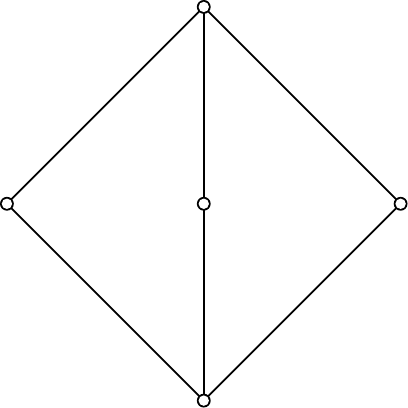}\rule{3em}{0ex}\includegraphics{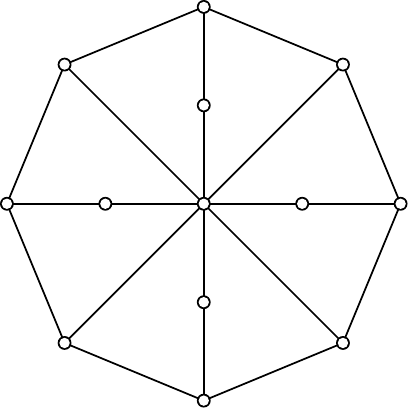}
  \end{center}
  \caption{\label{fig}A series-parallel graph and an ISK4-free graph}
\end{figure}

A \emph{bad triple} is a triple $(G, x, y)$ such that the following hold:
\begin{itemize}
\item $G$ is a graph and $x$ and $y$ are vertices of $G$.
\item $G$ does not have the $(x, y)$-property. 
\item If $x\neq y$, then $G$ has the $x$-property and $G$ has the
  $y$-property.
\end{itemize}

It is clear that for every graph $G$ and all vertices $x, y \in
V(G)$ such that $x=y$ or $xy\in E(G)$, one of the following cases
holds:
\begin{itemize}
\item $G$ has the $(x, y)$-property;
\item $(G, x, x)$ is a bad triple;
\item $(G, y, y)$ is a bad triple;
\item $(G, x, y)$ is a bad triple. 
\end{itemize}

Indeed, if $G$ does not have the $x$-property, then $(G, x, x)$ is a
bad triple.  Similarly, if $G$ does not have the $y$-property, then
$(G, y, y)$ is a bad triple.  So, we may assume that $G$ has the
$x$-property and the $y$-property.  If $G$ has the $(x, y)$-property,
then we are done, and otherwise, all the requirements in the
definition of bad triples are fulfilled.  It follows that a structural
description of bad triples really discribes all triples $(G, x, y)$
such that $G$ does not have the $(x, y)$-property.  Such a description
is given in the next lemma for triangle-free 2-connected
series-parallel graphs with no clique cutset.  Note that in the next
lemma, we do not require the girth of the graph to be at least~5.

\begin{lemma}
  \label{l:desc}
  Let $(G, x, y)$ be a bad triple, and suppose that $G$ is
  triangle-free, 2-connected, series-parallel, has least 5 vertices,
  and has no clique cutset.  Then, $G$ can be constructed as follows
  (and conversely, all graphs constructed as follows are
  triangle-free, 2-connected, series-parallel, have least 5 vertices,
  and have no clique cutset). 
  \begin{itemize}
  \item If $x\neq y$, then build two non-empty trees $T_x$ and $T_y$, not
    containing $x, y$, and consider the tree $T$ obtained by gluing
    $T_x$ and $T_y$ along some vertex $v$ (so $V(T_x) \cap V(T_y) =
    \{v\}$).
  \item If $x=y$ build a non-empty tree $T$, and set $T_x = T_y = T$.
  \item Add vertices of degree 2, each of them either adjacent to $x$ and
    to some vertex in $T_x$, or to $y$ and some vertex in $T_y$, in
    such a way that the following conditions are satisfied: 
    \begin{enumerate}
    \item $|N(x) \sm \{y\}|, |N(y) \sm \{x\}| \geq 1$;  
    \item every vertex of $T$ that has degree 2 in $T$ has at least
      one neighbor in $(N[x] \cup N[y]) \sm \{x, y\}$;
    \item every vertex of $T$ that has degree 1 in $T$ (so, a leaf of
      $T$) has at least two neighbors in $(N[x] \cup N[y])  \sm \{x, y\}$;
    \item every vertex of $T$ that has degree 0 in $T$ (this happens
      only when $|V(T)| = |V(T_x)| = |V(T_y)| = 1$) has at least three
      neighbors in $(N[x] \cup N[y])  \sm \{x, y\}$.
    \end{enumerate}
  \end{itemize}
\end{lemma}

\begin{proof}
  Let $(G, x, y)$ be a bad triple, and suppose that $G$ is
  triangle-free, 2-connected, series-parallel, has least 5 vertices,
  and has no clique cutset.
  
  \begin{claim}
    \label{c:chordless}
    $G$ is chordless.
  \end{claim}

  \begin{proofclaim}
    Suppose that some cycle $C$ of $G$ has a chord $ab$.  Thereofore,
    $C$ is formed of two $ab$-paths $R$ and $R'$, both of length at
    least~2.  Since $\{a, b\}$ is not a clique cutset of $G$, some
    path $S$ of $G$ is disjoint from $\{a, b\}$, has one end in $R$,
    the other one in $R'$, and is internally disjoint from $C$.
    Therefore, $C, ab$ and $S$ form a subdivision of $K_4$ (that is in
    $G$ as a subgraph), and this contradicts $G$ being
    series-parallel.    
  \end{proofclaim}

  \begin{claim}
    \label{c:Cdeg2}
    Every cycle of $G$ contains at least two
    vertices of degree~2.
  \end{claim}

  \begin{proofclaim}
    If $G$ is a cycle, our claim clearly holds.  Otherwise, it follows
    directly from~(\ref{c:chordless}) and Theorem~\ref{th:Cchordless}.
  \end{proofclaim}

  \begin{claim}
    \label{c:deg2}
    All vertices in $N(x) \cup N(y)$ distinct from $x$ and $y$ have
    degree 2.  
  \end{claim}

  \begin{proofclaim}
    Suppose w.l.o.g.\ that a neighbor of $x$, $x'\neq y$ has degree at
    least~3.  Let $u$ and $v$ be two neighbors of $x'$ distinct from
    $x$.  Because $G$ is triangle-free, $u$ and $v$ are non-adjacent
    to $x$.  Since $G$ is 2-connected, there is a path $P = u\dots v$
    in $G \sm x'$.  So, $C = x' u P v x'$ is a cycle of $G$.  Note
    that $C$ does not go through $x$, for otherwise $xx'$ would be a
    chord a $C$, a contradiction
    to~(\ref{c:chordless}).

    By~(\ref{c:Cdeg2}), $C$ contains two vertices $a$ and $b$ of
    degree~2.  Vertices $a$ and $b$ must be adjacent to $x$ or $y$
    because they have degree~2 and $G$ does not have the $(x,
    y)$-property.  If $a$ and $b$ are both adjacent to $x$ (in
    particular, when $x=y$), then $C$ must go through $x$ (because $a$
    and $b$ have degree~2), a contradiction.

    Hence, at least one of $a$ or $b$ is adjacent to $y$, not to $x$,
    and in particular, $x\neq y$.  If $C$ does not go through $y$,
    then $C$, $x$ and $y$ form a subdivision of $K_4$, contradiciting
    $G$ being series-parallel, so $C$ goes through $y$.  If $x$ is
    adjacent to $a$ or $b$, then again $C$, $x$ and $y$ form a
    subdivision of $K_4$.  So, $a$ and $b$ are both adjacent to $y$,
    not to $x$.  Since $G$ has the $y$-property (from the definition
    of bad triples), there is a non-neighbor $x''$ of $y$ that has
    degree~2, and since $G$ does not have the $(x, y)$-property, $x''$
    is a neighbor of $x$, distinct from $x'$ because $x'$ has
    degree~3.  Since $G$ is 2-connected, there exists a path $Q =
    x''\dots c$ from $x''$ to $C$ in $G \sm x$.  If $c \neq x', y$,
    then $x$, $Q$ and $C$ form a subdivision of $K_4$, a contradiction
    to $G$ being series-parallel.  Otherwise, $c \in \{x', c\}$ and
    $xc$ is a chord of some cycle of $G$, a contradiction
    to~(\ref{c:chordless}).
   \end{proofclaim}

  \begin{claim}
    \label{c:nonEmpty}
    ~$N[x]\cup N[y] \subsetneq V(G)$.
  \end{claim}

  \begin{proofclaim}
    Otherwise, $V(G) = N[x]\cup N[y]$  and all vertices of $G$ are
    adjacent to $x$ or $y$.  If $x=y$, then $V(G) = N[x]$, a
    contradiction since $G$ is triangle-free and 2-connected.  So,
    $x\neq y$.  

    Let $x'\neq y$ be a neighbor of $x$. By~(\ref{c:deg2}), $x'$ has degree~2.
    Its other neighbor $y'$ is non-adjacent to
    $x$ (because $G$ is triangle-free), so it must be adjacent to $y$,
    and by~(\ref{c:deg2}), $y'$ has degree~2.  Since $G$ contains at least
    five vertices, there must be other vertices, so w.l.o.g. another
    neighor $x''$ of $x$.  Again, $x''$ has degree~2, a neighbor $y''$
    (distinct from $y'$ because $y'$ has degree~2), and $y''$ is
    adjacent yo $y$.  Now, $xx'y'yy''x''x$ is a cycle of $G$ that has
    a chord (namely $xy$), a contradiction to~(\ref{c:chordless}).
  \end{proofclaim}

  \begin{claim}
    \label{c:stable}
    ~$(N[x]\cup N[y])  \sm \{x, y\}$ is a stable set.
  \end{claim}

  \begin{proofclaim}
    Otherwise, let $u$ and $v$ be two adjacent vertices in  $(N(x)\cup N(y))
    \sm \{x, y\}$.   By~(\ref{c:deg2}) they have degree~2, so $\{x,
    y\}$ is a clique cutset that separates them from $V(G) \sm
    (N[x]\cup N[y])$ that is non-empty by~(\ref{c:nonEmpty}). 
  \end{proofclaim}

  \begin{claim}
    \label{c:tree}
    ~$G \sm (N[x] \cup N[y])$ is a tree $T$.
  \end{claim}

  \begin{proofclaim}
    Since $(G, x, y)$ is a bad triple, $V(G) \sm (N[x] \cup N[y])$
    contains only vertices of degree at least~3 (in $G$), so
    by~(\ref{c:Cdeg2}), it cannot contain a cycle.  Also $G \sm
    (N[x] \cup N[y])$ is connected, because otherwise let $A$ and $B$
    be two components of $V(G) \sm (N[x] \cup N[y])$.  These two
    components must attach to disjoint sets of neighbors of
    $x$ and $y$, because all neighbors of $x$ and $y$ (except $x$ and
    $y$) have degree~2 by~(\ref{c:deg2}).  It follows that $\{x, y\}$
    is a clique cutset, a contradiction.
  \end{proofclaim}

  Let us now give label $x$ (resp.\ $y$) to all vertices of $T$ that
  have a neighbor adjacent to $x$ (resp.\ $y$).  Let us apply
  Lemma~\ref{l:tree} to $T$.

  If the first outcome holds (so in $T$, there exist two
  vertex-disjoint paths $P$ and $Q$, and each of them is from a vertex
  with label $x$ to a vertex with label $y$), then we reach a
  contradiction as follows.  Let $P = p_x \dots p_y$ and $Q = q_x
  \dots q_y$ where $p_x, q_x$ have label $x$ and $p_y, q_y$ have label
  $y$.  So, let $p'_x$ be a neighbor of $p_x$ that is adjacent to $x$,
  and let $p'_y$, $q'_x$ and $q'_y$ be defined similarly.  Observe
  that $p'_x$, $p'_y$, $q'_x$ and $q'_y$ are distinct, because all
  neighbors of $x$ and $y$ distinct from $x$ and $y$ have degree~2
  by~(\ref{c:deg2}).  Now the cycle $x p'_x p_x P p_y p'_y y q'_y q_y
  Q q_x q'_x x$ has a chord (namely $xy$), a contradiction
  to~(\ref{c:chordless}).
  
  Hence, the second outcome holds, so we keep the notation $v$, $T_x$
  and $T_y$ from Lemma~\ref{l:tree}.  Note that $T$ can be obtained by
  gluing $T_x$ and $T_y$ along~$v$.  It follows that $G$ can be
  constructed as we claim it should be. By~(\ref{c:tree}) and~(\ref{c:nonEmpty}), we really
  need to consider a non-empty tree.
  By~(\ref{c:deg2}), we have to add vertices of degree~2, and
  by~(\ref{c:stable}), they all have one neighbor in $\{x, y\}$ and
  the other one in $T$.  The last three conditions are here to ensure
  that the vertices of $T$ really all have degree at least 3.

  We do not prove the converse statement (every graph constructed as
  above is a bad triple, is triangle-free, 2-connected, series
  parallel, has at least 5 vertices and has no clique cutset).  It is
  easy to check and we do not need it in what follows. 
\end{proof}

\begin{lemma}\label{xy}
  Let $G$ be a 2-connected series-parallel graph, of girth at least~5 
  that has no clique cutset.  If $x$ and $y$ are vertices of $G$ such
  that $x=y$ or $xy\in E(G)$, then $G$ has the $(x, y)$-property.  
\end{lemma}

\begin{proof}
  Otherwise, one of $(G, x, y)$, $(G, x, x)$ or $G(y, y)$ is a bad
  triple.   We apply Lemma~\ref{l:desc}, and we consider the tree $T$
  defined in the outcome.  

  If $|V(T)| = 1$, then the unique vertex
  of $T$ has at least three neighbors in $(N[x]\cup N[y]) \sm \{x,
  y\}$, and at least two of them are neighbors of $x$, or are
  neighbors of $y$.  Therefore, $G$ contains a 4-cycle, a
  contradiction to our assumption on the girth.  

  If $|V(T)|  > 1$, then we consider a leaf of $T$ that is distinct
  from the vertex $v$ defined in the outcome of Lemma~\ref{l:desc}.
  This leaf has at least two neighbors that are both neighbors of $x$, or
  that are both neighbors of $y$.  Again, there exists a 4-cycle in
  $G$.  
\end{proof}

\begin{lemma}
  \label{l:2-conn}
  Let $G$ be a 2-connected ISK4-free graph of girth at least~5.  Then
  for every pair $\{x, y\}$ of vertices of $G$ such that $x=y$ or $xy
  \in E(G)$, $G$ has the $(x, y)$-property.
\end{lemma}

\begin{proof}
  We prove the statement by induction on $|V(G)|$.  If $G$ has no star
  cutset, then it has no clique cutset, and it is series parallel by
  Corollary~\ref{cmain} (all this happens in particular when $|V(G)|
  \leq 5$ which is the base case of our induction).  So, we have the
  result directly by Lemma~\ref{xy}.  Hence, we may assume that $G$
  has a star cutset $C$.  We suppose that $C$ is inclusionwise minimal
  among all possible star cutsets and is centered at $c$.  W.l.o.g.\ we suppose
  that $x$ and $y$ are both in $G[C \cup X]$ where $X$ is a component
  of $G\sm C$, and  we consider another component $Y$. 

  We claim that $G[C \cup Y]$ is 2-connected.  Indeed, suppose for a
  contradiction that $G[C \cup Y]$ has a cutvertex $v$.  Since $G$ is
  2-connected, $|C| \geq 2$.  By the minimality of $C$, every
  vertex of $C \sm \{c\}$ has a neighbor in $Y$ and if $|C| = 2$, then
  $c$ also has neighbors in $Y$.  It follows that $v\notin C$.  So,
  $v$ is in fact a cutvertex of $G$, a contradiction.  We proved  that
  $G[C \cup Y]$ is 2-connected. 

  We now apply the induction hypothesis to  $\{x, y, c\} \sm
  X$ in the graph $G[C \cup Y]$.  This gives a vertex in $G \sm (N[x]
  \cup N[y])$ that has degree~2 in $G$.
\end{proof}

\begin{theorem}
  \label{th:color}
  Every ISK4-free graph of girth at least~5
  contains a vertex of degree at most~2 and is 3-colorable.
\end{theorem}

\begin{proof}
  It is enough to prove that every ISK4-free graph of girth at least~5
  contains a vertex of degree at most~2.  For the sake of induction,
  we prove by induction on $|V(G)|$ a slightly stronger statement:
  every ISK4-free graph of girth at least~5 on at least two vertices
  contains at least two vertices of degree at most~2.  If $|V(G)| =
  2$, this is clearly true.  If $G$ is 2-connected, it follows from
  Lemma~\ref{l:2-conn} applied twice (once to find a vertex $x$ of
  degree~2, and another time to find the second one in $G \sm N[x]$).
  So, we may assume that $G$ is not 2-connected and has at least~3
  vertices, so it has a cutvertex~$v$.  The result follows from the
  induction hypothesis applied to $G[X \cup \{v\}]$ and to $G[Y \cup
  \{v\}]$ where $X$ and $Y$ are connected components of $G \sm \{v\}$.
\end{proof}

\end{document}